\documentclass[a4paper,12pt]{amsart}
\usepackage{amssymb,amsmath,amscd,tikz,url,hyperref}
\usepackage[all]{xy}

\newtheorem{thm}{Theorem}[section]

\newtheorem{cor}[thm]{Corollary}

\newtheorem{prop}[thm]{Proposition}

\newtheorem{ex}[thm]{Example}

\newtheorem{defn}[thm]{Definition}

\numberwithin{equation}{subsection}
\numberwithin{equation}{section}

\def\Irr{\mathbf{Irr}}
\def\cG{\mathcal{G}}
\def\cB{\mathcal{B}}
\def\fL{\mathfrak{L}}
\def\cS{\mathcal{S}}
\def\cT{\mathcal{T}}

\def\fg{\mathfrak{g}}
\def\fs{\mathfrak{s}}

\def\fT{\mathfrak{T}}
\def\fsl{\mathfrak{sl}}
\def\fso{\mathfrak{so}}
\def\Ad{\mathrm{Ad}}
\def\Deg{\mathrm{Deg}}
\def\dim{\mathrm{dim}}
\def\Gal{\mathrm{Gal}}
\def\Ind{\mathrm{Ind}}

\def\temp{\mathrm{temp}}
\def\St{\mathrm{St}}

\def\SO{\mathrm{SO}}

\def\GL{\mathrm{GL}}
\def\Aut{\mathrm{Aut}}
\def\F{\mathbb{F}}

\def\C{\mathbb{C}}
\def\Z{\mathbb{Z}}
\def\T{\mathbb{T}}
\def\Q{\mathbb{Q}}
\def\SL{\mathrm{SL}}
\def\PSL{\mathrm{PSL}}
\def\im{\mathrm{im}}
\def\triv{\mathrm{triv}}
\def\Prim{\mathbf{Prim}}
\def\W{\mathbf{W}}
\def\q{/\!/}
\def\phi{\varphi}

\begin{document}

\title[Formal degrees]{$L$-packets and formal degrees for $\SL_2(K)$ with $K$ a local function field of characteristic $2$}

\author[S. Mendes]{Sergio Mendes}
\address{ISCTE - Lisbon University Institute, Av. das For\c{c}as Armadas, 1649-026, Lisbon, Portugal}
\email{sergio.mendes@iscte.pt}
\author[R. Plymen]{Roger Plymen}
\address{School of Mathematics, Southampton University, Southampton SO17 1BJ,  England \emph{and} School of Mathematics, Manchester University,
Manchester M13 9PL, England}
\email{r.j.plymen@soton.ac.uk \quad plymen@manchester.ac.uk}

%\date{\today}
\maketitle

\begin{abstract}  Let $\cG = \SL_2(K)$ with $K$ a local function field of characteristic $2$.   We review  Artin-Schreier theory for the
field $K$,  and show that this leads to a parametrization of $L$-packets in the smooth dual of $\cG$.  We relate this to a recent geometric conjecture.
The $L$-packets in the principal series are parametrized by quadratic extensions, and the supercuspidal $L$-packets by biquadratic extensions.
We compute the formal degrees of the elements in the supercuspidal packets.

\end{abstract}

\tableofcontents

\section{Introduction}

In this article we consider  a local function field $K$ of characteristic $2$, namely
$K= \F_q((\varpi))$,  the field of Laurent series with
coefficients in $\F_q$, with $q=2^f$. This example is particularly
interesting because there are countably many quadratic extensions of
$\F_q((\varpi))$.

 We consider $\cG = \SL_2(K)$.
 %   over the local function field $\F_q((\varpi))$.
  Drawing on the accounts in \cite{Da, Th1, Th2}, we review  Artin-Schreier theory, adapted to the local function
field $ \F_q((\varpi))$.   This leads to a parametrization of $L$-packets in the smooth dual of $\cG$.   In this article, we reserve the term $L$-packets for the ones which are not singletons.

The $L$-packets in the principal series are parametrized by quadratic extensions, and the supercuspidal $L$-packets by biquadratic extensions. There are countably many supercuspidal packets.

By \emph{canonical formal degree} we shall mean formal degree with respect to the Euler-Poincar\'e measure on $\cG$, as in \cite{Reed}.   We compute the canonical formal degrees of the elements in the supercuspidal packets,  relying on the Formal Degree component of the local Langlands correspondence, see  \cite[\S6]{Reed}.   The canonical formal degrees are all dyadic rationals, in fact they are integer powers of $2$.   They depend on the residue degree $f$, and on the breaks in the lower ramification filtration of the Galois group $\Z/2\Z \times \Z/2\Z$.

The commutative triangle in Theorem 4.3, and the bijective maps \ref{T1}, \ref{T2}, \ref{T3} in \S5, amount to a proof, for $\cG$,  of the \emph{tempered} version of the geometric conjecture in \cite{ABPS}.

Thanks to Anne-Marie Aubert for her careful reading of the file, and for sending us some valuable comments.
The first author would like to thank Chandan Dalawat for a valuable exchange of emails and for the reference \cite{Da}.

\section{Artin-Schreier theory}

Let $K$ be a local field with positive characteristic $p$. The cyclic
extensions of $K$ whose degree $n$ is coprime with $p$ are described
by Kummer theory. It is well known that any cyclic
extension $L/K$ of degree $n$, $(n,p)=1$, is generated by a root
$\alpha$ of an irreducible polynomial $x^n-a\in K[x]$.
If $\alpha\in K^s$ is a root of $x^n-a$ then $K(\alpha)/K$ is a
cyclic extension of degree $n$ and is called a Kummer extension of
$K$.

Artin-Schreier theory aims to describe cyclic extensions of degree equal to
or divisible by $ch(K)=p$. It is therefore an analogue of Kummer theory, where the
role of the polynomial $x^n-a$ is played by $x^n-x-a$. Essentially,
every cyclic extension of $K$ with degree $p=ch(K)$ is generated by
a root $\alpha$ of $x^p-x-a\in K[x]$.

We fix an algebraic closure $\overline{K}$ of $K$ and a separable closure
$K^s$ of $K$ in $\overline{K}$. Let $\wp$ denote the Artin-Schreier endomorphism of the additive group
$K^s$ \cite{Ne}:
\[
\wp : K^s \to K^s, \quad x \mapsto x^p-x.
\]

Given $a\in K$ denote by $K(\wp^{-1}(a))$ the extension $K(\alpha)$,
where $\wp(\alpha)=a$ and $\alpha\in K^s$. We have the following characterization of finite cyclic Artin-Schreier extensions of degree $p$:

\begin{thm}
\begin{itemize}
\item[$(i)$] Given $a\in K$, either $\wp(x)-a\in K[x]$ has one root in $K$ in which case it has all the $p$ roots are in $K$, or is irreducible.
\item[$(ii)$] If $\wp(x)-a\in K[x]$ is irreducible then $K(\wp^{-1}(a))/K$ is a cyclic extension of degree $p$, with $\wp^{-1}(a)\subset K^s$.
\item[$(iii)$] If $L/K$ be a finite cyclic extension of degree $p$, then $L=K(\wp^{-1}(a))$, for some $a\in K$.
\end{itemize}
\end{thm}
(See \cite[p.34]{Th1} for more details.)

\bigskip

We fix now some notation. $K$ is a local field with characteristic $p>1$ with finite
residue field $k$. The field of constants $k=\F_q$ is a finite extension of $\F_p$, with degree $[k:\F_p]=f$ and $q=p^f$.

Let $\mathfrak{o}$ be the ring of integers in $K$ and denote by $\mathfrak{p}\subset\mathfrak{o}$ the (unique) maximal ideal of $\mathfrak{o}$.
This ideal is principal and any generator of $\mathfrak{p}$ is called a uniformizer. A choice of uniformizer $\varpi\in\mathfrak{o}$ determines isomorphisms $K\cong\F_q((\varpi))$, $\mathfrak{o}\cong\F_q[[\varpi]]$ and $\mathfrak{p}=\varpi\mathfrak{o}\cong\varpi\F_q[[\varpi]]$.

A normalized valuation on $K$ will be denoted by $\nu$, so that $\nu(\varpi)=1$ and $\nu(K)=\Z$. The group of units is denoted by $\mathfrak{o}^{\times}$.

\subsection{The Artin-Schreier symbol}

%An extension $L$ of $K$ is called $p$-closed if it has no Galois extensions of degree $p$ \cite[p.290]{NSW}. Examples of
%$p$-closed extensions are the separable closure $K^s/K$ and $K_p/K$. The endomorphism $\wp\in End(K)$ extends to
%any extension of $K$. In particular, it extends to any $p$-closed extension.

Let $L/K$ be a finite Galois extension. Let $N_{L/K}$ be the norm map and denote $G_{L/K}^{ab}=Gal(L/K)^{ab}$ the abelianization of $Gal(L/K)$. The reciprocity map is a group isomorphism
\begin{equation}
K^{\times}/N_{L/K} L^{\times}\stackrel{\simeq}\longrightarrow G_{L/K}^{ab}.
\end{equation}\label{Artin symbol}
The Artin symbol is obtained by composing the reciprocity map with the canonical morphism $K^{\times}\rightarrow K^{\times}/N_{L/K} L^{\times}$
\begin{equation}
b\in K^{\times} \mapsto(b,L/K)\in G_{L/K}^{ab}.
\end{equation}

From the Artin symbol we obtain a pairing
\begin{equation}\label{A-S pairing KxK*}
K\times K^{\times}\longrightarrow\mathbb{Z}/p\mathbb{Z} , (a,b)\mapsto(b,L/K)(\alpha)-\alpha,
\end{equation}
where $\wp(\alpha)=a$, $\alpha\in K^s$ and $L=K(\alpha)$.

\begin{defn}
Given $a\in K$ and $b\in K^{\times}$, the Artin-Schreier symbol is defined to be
\[
[a,b)=(b,L/K)(\alpha)-\alpha.
\]
\end{defn}

We summarize some important properties of the Artin-Schreier symbol.

\begin{prop}\label{Artin-Schreier symbol}
The Artin-Schreier symbol is a bilinear map satisfying the following properties:
\begin{itemize}
\item[$(i)$] $[a_1+a_2,b)=[a_1,b)+[a_2,b)$;
\item[$(ii)$] $[a,b_1b_2)=[a,b_1)+[a,b_2)$;
\item[$(iii)$] $[a,b)=0, \forall a\in K \Leftrightarrow b\in N_{L/K}L^{\times}, L=K(\alpha) \textrm{ and }\wp(\alpha)=a$;
\item[$(iv)$] $[a,b)=0, \forall b\in K^{\times} \Leftrightarrow a\in\wp(K)$.
\end{itemize}
\end{prop}
(See \cite[p.341]{Ne})

\subsection{The groups $K/\wp(K)$ and $K^{\times}/K^{\times p}$}

In this section we recall some properties of the groups $K/\wp(K)$ and $K^{\times}/K^{\times p}$ and use them to redefine the pairing (\ref{A-S pairing KxK*}).

Consider the additive group $K$. The index of $\wp(K)$ in $K$ is infinite \cite[p.146]{FV}. Hence, $K/\wp(K)$ is infinite.
\begin{prop}\label{group K/PK}
$K/\wp(K)$ is a discrete abelian torsion group.
\end{prop}
\begin{proof}
The ring of integers decomposes as a (direct) sum
\[
\mathfrak{o}=\mathbb{F}_q+\mathfrak{p}
\]
and we have
\[
\wp(\mathfrak{o})=\wp(\mathbb{F}_q)+\wp(\mathfrak{p}).
\]
The restriction $\wp:\mathfrak{p}\rightarrow\mathfrak{p}$ is an isomorphism, see \cite[Lemma 8]{Da}. Hence,
\[
\wp(\mathfrak{o})=\wp(\mathbb{F}_q)+\mathfrak{p}
\]
and $\mathfrak{p}\subset\wp(K)$. It follows that $\wp(K)$ is an open subgroup of $K$ and $K/\wp(K)$ is discrete. Since $\wp(K)$ is annihilated by $p$, $K/\wp(K)$ is a torsion group.
\end{proof}

Now we concentrate on the multiplicative group $K^{\times}$. The subgroup $K^{\times p}$ is not open in $K^{\times}$ and the index $[K^{\times}:K^{\times p}]$ is infinite \cite[Lemma p.115]{FV}. Hence, $K^{\times}/K^{\times p}$ is infinite. The next result gives a characterization of the topological group $K^{\times}/K^{\times p}$.

\begin{prop}\label{K*/K^*p is profinite}
$K^{\times}/K^{\times p}$ is a profinite abelian $p$-torsion group.
\end{prop}
\begin{proof}
There is a canonical isomorphism $K^{\times}\cong\mathbb{Z}\times\mathfrak{o}^{\times}$. By \cite[p.25]{Iw}, the group of units $\mathfrak{o}^{\times}$ is a direct product of countable many copies of the ring of $p$-adic integers
$$\mathfrak{o}^{\times}\cong\mathbb{Z}_p\times\mathbb{Z}_p\times\mathbb{Z}_p\times...=\prod_{\mathbb{N}}\mathbb{Z}_p.$$
Give $\mathbb{Z}$ the discrete topology and $\mathbb{Z}_p$ the $p$-adic topology. Then, for the product topology, $K^{\times}=\mathbb{Z}\times\prod_{\mathbb{N}}\mathbb{Z}_p$ is a topological group, locally compact, Hausdorff and totally disconnected.

$K^{\times p}$ decomposes as a product of countable many components
\[
K^{\times p}\cong p\mathbb{Z}\times p\mathbb{Z}_p\times p\mathbb{Z}_p\times...=p\mathbb{Z}\times\prod_{\mathbb{N}}p\mathbb{Z}_p.
\]

Denote by $y=\prod_n y_n$ and element of $\prod_{\mathbb{N}}\mathbb{Z}_p$, where $y_n=\sum_{i=0}^{\infty}a_{i,n}p^i \in\mathbb{Z}_p$, for every $n$.

The map
\[
\varphi:\mathbb{Z}\times\prod_{\mathbb{N}}\mathbb{Z}_p\rightarrow\mathbb{Z}/p\mathbb{Z}\times\prod_{\mathbb{N}}\mathbb{Z}/p\mathbb{Z}, (x,y)\mapsto (x(mod p),\prod_n pr_0(y_n))
\]
where $pr_0(y_n)=a_{0,n}$ is the projection, is clearly a group homomorphism.

Now, $\mathbb{Z}/p\mathbb{Z}\times\prod_{\mathbb{N}}\mathbb{Z}/p\mathbb{Z}=\prod_{n=0}^{\infty}\mathbb{Z}/p\mathbb{Z}$ is a topological group for the product topology, where each component $\mathbb{Z}/p\mathbb{Z}$ has the discrete topology. Moreover, it is compact by Tyconoff Theorem, Hausdorff and totally disconnected \cite[TGI.84, Prop. 10]{Bo}. Therefore, $\prod_{n=0}^{\infty}\mathbb{Z}/p\mathbb{Z}$ is a profinite group.

Since
\[
ker\varphi=p\mathbb{Z}\times\prod_{\mathbb{N}}p\mathbb{Z}_p,
\]
it follows that there is an isomorphism of topological groups
\[
K^{\times}/K^{\times p}\cong\prod_{\mathbb{N}}p\mathbb{Z}_p,
\]
where $K^{\times}/K^{\times p}$ is given the quotient topology. Therefore, $K^{\times}/K^{\times p}$ is profinite.
\end{proof}

From Propositions \ref{Filtration K/P(K)} and \ref{K*/K^*p is profinite}, $K/\wp(K)$ is a discrete abelian group and $K/K^{\times p}$ is an abelian profinite group, both annihilated by $p=ch(K)$. Therefore, Pontryagin duality coincides with $Hom(-,\mathbb{Z}/p\mathbb{Z})$ on both of these groups, see \cite{Th2}. See also \cite{RZ} for more details on Pontryagin duality. The pairing (\ref{A-S pairing KxK*}) restricts to a pairing

\begin{equation}\label{A.-S- Pairing}
[.,.): K/\wp(K)\times K^{\times}/K^{\times p}\rightarrow\Z/p\Z.
\end{equation}
which we refer from now on to the \textbf{Artin-Schreier pairing}. It follows from $(iii)$ and $(iv)$ of Proposition \ref{Artin-Schreier symbol}, the pairing is nondegenerate (see also \cite[Proposition 3.1]{Th2}). The next result shows that the pairing is perfect.

\begin{prop}\label{A-S symbol induce quadratic character}
The Artin-Schreier symbol induces isomorphisms of topological groups
\[
K^{\times}/K^{\times p}\stackrel{\simeq}\longrightarrow Hom(K/\wp(K),\mathbb{Z}/p\mathbb{Z}), bK^{\times p}\mapsto(a+\wp(K)\mapsto[a,b))
\]
and
\[
K/\wp(K)\stackrel{\simeq}\longrightarrow Hom(K^{\times}/K^{\times p},\mathbb{Z}/p\mathbb{Z}), a+\wp(K)\mapsto(bK^{\times p}\mapsto[a,b))
\]
\end{prop}
\begin{proof}
The result follows by taking $n=1$ in Proposition $5.1$ of \cite{Th2}, and from the fact that Pontryagin duality for the groups $K/\wp(K)$ and $K^{\times}/K^{\times p}$ coincide with $Hom(-,\mathbb{Z}/p\mathbb{Z})$ duality. Hence, there is an isomorphism of topological groups between each such group and its bidual.
\end{proof}

Let $B$ be a subgroup of the additive group of $K$ with finite index such that $\wp(K)\subseteq B\subseteq K$. The composite of two finite abelian Galois extensions of exponent $p$ is again a finite abelian Galois extension of exponent $p$. Therefore, the composite
\[
K_B=K(\wp^{-1}(B))=\prod_{a\in B}K(\wp^{-1}(a))
\]
is a finite abelian Galois extension of exponent $p$. On the other hand, if $L/K$ is a finite abelian Galois extension of exponent $p$, then $L=K_B$ for some subgroup $\wp(K)\subseteq B\subseteq K$ with finite index.

All such extensions lie in the maximal abelian extension of exponent $p$, which we denote by $K_p=K(\wp^{-1}(K))$. The extension $K_p/K$ is infinite and Galois. The corresponding Galois group $G_p=Gal(K_p/K)$ is an infinite profinite group and may be identified, under class field theory, with $K^{\times}/K^{\times p}$, see \cite[Proposition 5.1]{Th2}. The case $ch(K)=2$ leads to $G_2\cong K^{\times}/K^{\times 2}$ and will play a fundamental role in the sequel.

%\bigskip
%By Proposition \ref{K*/K^*p is profinite} and Proposition \ref{A-S symbol induce quadratic character} we conclude that Artin-Schreier theory gives an %isomorphism of topological groups (see \cite{Th2})
%\begin{equation}\label{A-S isomorphism}
%\mathfrak{as}:G_p \stackrel{\simeq}\longrightarrow Hom(K/\wp(K), \mathbb{Z}/p\mathbb{Z}).
%\end{equation}
%Explicitly, $\sigma\mapsto\phi_{\sigma}$,where $\phi_{\sigma}=(a+\wp(K)\mapsto\sigma(\alpha)-\alpha)$, for some $\alpha\in K^s$ such that %$\wp(\alpha)=a$.

\section{Quadratic characters}   From now on we take $K$ to be a local function field with $ch(K)=2$.
%nsider the ground field to be $K=\mathbb{F}_q((\varpi))$, where $q=2^f$.
Therefore, $K$ is of the form $\mathbb{F}_q((\varpi))$ with $q = 2^f$.
%a local function field of Laurent series with coefficients in $\mathbb{F}_q$.

Recall that a character of $K^{\times}$ is a group homomorphism
\[
\chi:K^{\times}\to \T
\]
where $\T = \{z \in \C : |z| = 1\}$ is the unit circle.
Denote by $\widehat{K^{\times}}$ the group of characters of $K^{\times}$. There is a canonical isomorphism
\[
\widehat{K^{\times}}\cong\widehat{\mathbb{Z}\times \mathfrak{o}^{\times}}\cong\mathbb{T}\times\widehat{\mathfrak{o}^{\times}}.
\]
 Therefore, given a character $\chi\in\widehat{K^{\times}}$, we may write
$\chi=z^{\nu}\chi_0$, where $z\in\mathbb{T}$, $\nu$ is the valuation and
$\chi_0\in\widehat{\mathfrak{o}^{\times}}$. If $\chi_0\equiv 1$ we say that $\chi$ is
unramified. A character $\chi$ of $K^{\times}$ is called quadratic if $\chi^2=1$. Since the unique quadratic character of $\mathbb{Z}$ is $(n\mapsto(-1)^n)$, a nontrivial quadratic character has the form $\chi=(-1)^{\nu}\chi_0$, with $\chi_0^2=1$.

When $K = \mathbb{F}_q((\varpi))$, we have,  according to \cite[p.25]{Iw},
\[
\mathfrak{o}^{\times} \cong \mathbb{Z}_2\times\mathbb{Z}_2\times\mathbb{Z}_2\times...=\prod_{\mathbb{N}}\mathbb{Z}_2
\]
with countably infinite many copies of $\mathbb{Z}_2$, the ring of $2$-adic integers.

Artin-Schreier theory provides a way to parametrize all the
quadratic extensions of $K=\mathbb{F}_q((\varpi))$. By Proposition \ref{K*/K^*p is profinite}, there is a
bijection between the set of quadratic extensions of
$\mathbb{F}_q((\varpi))$ and the group
\[
\mathbb{F}_q((\varpi))^{\times}/\mathbb{F}_q((\varpi))^{\times 2}\cong\prod_{\mathbb{N}}2\mathbb{Z}_2= G_2
\]
where $G_2$ is the Galois group of the maximal abelian extension of
exponent $2$. Since $G_2$ is an infinite profinite group, there are
countably many quadratic extensions.

To each quadratic extension $K(\alpha)/K$, with $\alpha^2-\alpha=a$, we associate the Artin-Schreier symbol
\[
[a,.) : K^{\times}/K^{\times 2}\rightarrow\mathbb{Z}/2\mathbb{Z}.
\]
Now, let $\varphi$ denote the isomorphism $\mathbb{Z}/2\mathbb{Z}\cong\mu_2(\mathbb{C})=\{\pm 1\}$ with the group of roots of unity. We obtain, by composing with the Artin-Schreier symbol, a unique multiplicative quadratic character $\chi_a=\varphi([a,.))$:

\begin{equation}\label{def A-S character}
\chi_a: K^{\times}\rightarrow\mathbb{C}^{\times}.
\end{equation}

Proposition \ref{A-S symbol induce quadratic character} shows that every quadratic character of $\mathbb{F}_q((\varpi))^{\times}$ arises in this way.

\begin{ex}\label{example: unr. quadratic ext.}
The unramified quadratic extension of $K$ is $K(\wp^{-1}(\mathfrak{o}))$, see \cite{Da} proposition $12$. According to Dalawat, the group $K/\wp(K)$ may be regarded as an $\mathbb{F}_2$-space and the image of $\mathfrak{o}$ under the canonical surjection $K\rightarrow K/\wp(K)$ is an $\mathbb{F}_2$-line, i.e., isomorphic to $\mathbb{F}_2$. Since $\wp_{|\mathfrak{p}}:\mathfrak{p}\rightarrow\mathfrak{p}$ is an isomorphism, the image of $\mathfrak{p}$ in $K/\wp(K)$ is $\{0\}$, see lemma $8$ in \cite{Da}. Now, choose any $a_0\in\mathfrak{o}\backslash\mathfrak{p}$. The quadratic character $\chi_{a_0}=\varphi([a_0,.))$ associated with $K(\wp^{-1}(\mathfrak{o}))$ via class field theory is precisely the unramified character $(n\mapsto(-1)^n)$ from above. Note that any other choice $b_0\in\mathfrak{o}\backslash\mathfrak{p}$, with $a_0\neq b_0$, gives the same unique unramified character, since there is only one nontrivial coset $a_0+\wp(K)$ for $a_0\in\mathfrak{o}\backslash\mathfrak{p}$.
\end{ex}

\smallskip

Let $\cG$ denote  $\SL_2(K)$, let $\cB$ be the standard Borel subgroup of $\cG$, let $\cT$ be the diagonal subgroup of $\cG$.
Let $\chi$ be a character of $\cT$. Then, $\chi$ inflates to a character of $\cB$. Denote by $\pi_{\chi}$ the (unitarily) induced representation $Ind_{\cB}^{\cG}(\chi)$. The representation space of $V_{\chi}$ of $\pi_{\chi}$ consists of locally constant complex valued functions $f:\cG\rightarrow\mathbb{C}$ such that, for every $a\in K^{\times}$, $b\in K$ and $g\in \cG$, we have

\[
f\bigg(\left( \begin{array}{cc}
 a & b \\
 0 & a^{-1}
\end{array} \right)\Bigg)=|a|\chi(a)f(g)
\]

The action of $\cG$ on $V_{\chi}$ is by right translation. The representations $(\pi_{\chi},V_{\chi})$ are called (unitary) principal series of $\cG=SL_2(K)$.

Let $\chi$ be a quadratic character of $K^{\times}$. The reducibility of the induced representation $\Ind_B^G(\chi)$ is well known in zero characteristic. W. Casselman proved that the same result holds in characteristic $2$ and any other positive characteristic $p$.

\begin{thm}\cite{Ca,Ca2}\label{Casselman's th.}
The representation $\pi_{\chi}= \Ind_{\cB}^{\cG}(\chi)$ is reducible if, and only if, $\chi$ is either $|.|^{\pm}$ or a nontrivial quadratic character of $K^{\times}$.
\end{thm}

For a proof see \cite[Theorems 1.7, 1.9]{Ca} and \cite[\S 9]{Ca2}.

\medskip

From now on, $\chi$ will be a quadratic character. It is a classical result that the unitary
principal series for $\GL_2$ are irreducible. For a representation of $\GL_2$ parabolically induced by
$1\otimes\chi$, Clifford theory tells us that the dimension of the intertwining algebra of its restriction to
$\SL_2$ is $2$. This is exactly the induced representation of $\SL_2$ by $\chi$:
\[
\Ind_{\widetilde{B}}^{\GL_2(K)}(1\otimes\chi)_{|\SL(2,K)}\stackrel{\simeq}\longrightarrow \Ind_{B}^{\SL_2(K)}(\chi)
\]
where $\widetilde{B}$ denotes the standard Borel subgroup of $\GL_2(K)$. This leads to reducibility of the induced representation
$\Ind_B^G(\chi)$ into two inequivalent constituents.  Thanks to M. Tadic for helpful comments at this point.

The two irreducible constituents
\begin{align}
\pi_{\chi}= \Ind_{B}^{G}(\chi)=\pi_{\chi}^+\oplus\pi_{\chi}^-
\end{align}
define an $L$-packet $\{\pi_{\chi}^+ , \pi_{\chi}^-\}$ for $\SL_2$.

\section{A commutative triangle}

In this section we confirm part of the  geometric conjecture in \cite{ABPS} for $\SL_2(\mathbb{F}_q((\varpi)))$. We begin by recalling the underlying ideas of the conjecture.

Let $\cG$ be the group of $K$-points of a connected reductive group over a nonarchimedean local field $K$.   We have the \emph{Bernstein decomposition}
\[
\Irr(\cG)  = \bigsqcup \Irr(\cG)^{\fs}
\]
over all points $\fs \in \mathfrak{B}(\cG)$ the Bernstein spectrum of $\cG$, see \cite{R}.

Let $\chi_{a_0}=\varphi([a_0,.))$ denote the unramified character of $K^{\times}$ associated with the unramified quadratic extension $K(\alpha_0)=K(\wp^{-1}(\mathfrak{o}))$ as in example \ref{example: unr. quadratic ext.}. Fix a quadratic character $\chi_a=\varphi([a,.))$ associated via class field theory with the quadratic extension $K(\alpha)$ (in a fixed algebraic closure $\overline{K}$), where $\alpha^2-\alpha=a$.

\begin{prop}\label{compositum quadratic ext}
There is a unique quadratic extension $K(\beta)$ with associated character $\chi_{a_0+a}$. Moreover, $\chi_{a_0+a}=\chi_{a_0}\chi_a$.
\end{prop}
\begin{proof}
The compositum $K(\alpha)K(\alpha_0)$ is Galoisian over $K$, with Galois group $\mathbb{Z}/2\mathbb{Z}\times\mathbb{Z}/2\mathbb{Z}$. Therefore, contains three subfields, which are quadratic extensions of $K$, namely $K(\alpha_0)$, $K(\alpha)$ and, say, $K(\beta)$. The extension $K(\beta)$ is such that $\beta^2-\beta=a_0+a$, and has an associated quadratic character given by $\chi_{a_0+a}$. Hence
$$\chi_{a_0+a}=\varphi([a_0+a,.))=\varphi([a_0,.)+[a,.))=\varphi([a_0,.))\varphi([a,.))=\chi_{a_0}\chi_a.$$
\end{proof}

\medskip

By theorem \ref{Casselman's th.}, the induced representations
\[
\pi_{a_0}= \Ind_{\cB}^{\cG}(\chi_{a_0}) \textrm{ , }\pi_a= \Ind_{\cB}^{\cG}(\chi_a)\textrm{ and }\pi_{a_0+a}= \Ind_{\cB}^{\cG}(\chi_{a_0+a})
\]
are reducible and split into a direct sum of two irreducible component.

\bigskip

Central to the geometric conjecture is the concept of extended quotient of the second kind, which we now define.

Let $W$ be a finite group and let $X$ be a complex affine algebraic variety. Suppose that $W$ is acting on $X$ as automorphisms of $X$. Define
\[
\widetilde{X}_2:  =  \{(x,\tau) : \tau \in \Irr(W_x)\}.
\]
Then $W$ acts on $\widetilde{X}_2$:
\[
\alpha(x, \tau)=(\alpha \cdot x,\alpha_* \tau).
\]
\begin{defn}
The extended quotient of the second kind is defined as
\[
(X\q W)_2 :=\widetilde{X}_2/W.
\]
\end{defn}
Thus the extended quotient of the second kind is the ordinary quotient for the action of $W$ on $\widetilde{X}_2$.

\begin{thm} \label{Triangle}
Let $\cG = \SL_2(K)$ with $K = \F_q((\varpi)))$.   Let $\fs = [\cT,\chi]_G$ be a point in the Bernstein spectrum for the principal series of
$\cG$.    Let $\Irr(\cG)^{\fs}$ be the corresponding  Bernstein component in $\Irr(\cG)$. Then the conjecture \cite{ABPS} is valid for $\Irr(\cG)^{\fs}$
i.e. there is a commutative triangle of natural bijections
\[
\xymatrix{
& (T^{\fs}/\!/W^\fs)_2 \ar[dr]\ar[dl] & \\
\Irr(\mathcal{G})^\fs    \ar[rr] & & \fL(G)^\fs
%\{\mathrm{enhanced}\; L-\mathrm{parameters}\}^\fs/ H^\fs }
}
\]
where  $\fL(G)^\fs$ denotes the equivalence classes of enhanced parameters attached to $\fs$.
\end{thm}

\begin{proof}    We recall that $(G,T)$ are the complex dual groups of $(\cG,\cT)$.   Let $\W_K$ denote the Weil group of $K$.
%We will set $G = \PGL_2(\C)$.
  If $\phi$ is an $L$-parameter
  \[
  \W_K \times \SL_2(\C) \to G
  \]
   then $\mathcal{S}_{\phi}$ is defined as follows:
\[
\cS_{\phi}: = \pi_0 \, C_G(\im \, \phi).
\]

By an \emph{enhanced} Langlands parameter, we shall mean a pair $(\phi, \rho)$ where $\phi$ is a parameter and $\rho \in \Irr(\cS_{\phi})$.  Following Reeder
\cite{Reed}, we shall denote an enhanced Langlands parameter by $\phi(\rho)$.

\textbf{Case 1}.  Let $\chi$ be a quadratic character of $\cT$: $\chi^2 = 1, \chi \neq 1$.    Let $L/K$ be the quadratic extension
determined by  $\chi$.    Now $G$ contains a unique (up to conjugacy) subgroup
$H \simeq \Z/2\Z$.   Each quadratic extension $L/K$ creates a parameter
\[
\phi_L : \W_K \to \Gal(L/K) \to G.
\]
The map $\Gal(L/K) \to H$ factors through $K^{\times}/N_{L/K} L^{\times}$:
\[
\phi_L : \W_K \to \Gal(L/K) \simeq K^{\times}/N_{L/K} L^{\times} \to H \to G.
\]
which shows that $\phi_L$ is the parameter attached to the packet $\pi_{\chi}$.

%which is a non-discrete parameter.
% with $S_{\phi_L}	=  H$, since $C_G(H) = H$,  and whose conjugacy class depends only on $L$, since $O/H = \Aut(H)$.

To compute $\cS_{\phi_L}$, let $1,w$ be representatives of the Weyl group $W = W(G)$.   Then we have
\[
C_G(\im\,\varphi_L) = T \sqcup wT
\]
So $\phi$ is a non-discrete parameter, and we have
\[
\cS_{\phi_L} \simeq \Z/2\Z.
\]

We have two enhanced Langlands parameters, namely $\phi_L(\triv)$ and $\phi_L(\rho)$ where $\rho$ is the nontrivial character of $\cS_{\phi_L}$.

Now define
\[
\chi(\varpi) =
\chi \left(
\begin{array}{cc}
\varpi & 0\\
0 & \varpi^{-1}
\end{array}
\right)
\]
where $\varpi$ is a uniformizer in $K$.

%Let $\cT_0$ denote the maximal compact subgroup of $\cT$.  We write $\cT = \cT_0 \times X$ and factorize $\chi$ as follows $\chi = \chi_0 \otimes %\psi$ so that $\psi$ is an unramifed character.   Let $\varpi_K$ be a uniformizer in $K$.

Since $\chi^2 = 1$, there is a point of reducibility.   We have, at the level of elements,

%Let $\chi|\cT_0 \neq 1,\; \psi^2 = 1$,  two points of reducibility in the ramified unitary principal series.     In this case, we have, at the level of elements,

\[
\xymatrix{
& \{(\chi(\varpi), \triv), (\chi(\varpi), \rho)\} \ar[dr]\ar[dl] & \\
\{\pi_{\chi}^+, \pi_{\chi}^- \}    \ar[rr] & & \{\phi_{\chi}(\triv), \phi_{\chi}(\rho)\}}
\]

%Let $\chi|\mathfrak{o}_K  = 1, \;\psi^2 \neq 1$.   In this case, we have, at the level of elements, \[\xymatrix{& (\psi(\varpi_K), \triv) \ar[dr]\ar[dl] & \\
%\pi_{\psi}    \ar[rr] & & \phi_{\psi}}\]

%Let $\chi| \cT_0  = 1, \;\psi(x) = (-1)^{\val_K(x)}$, the unique point of reducibility in the unramified unitary principal series.   We have, at the level of %elements,\[\xymatrix{& \{(\psi(\varpi_K), \triv), (\psi(\varpi_K), \rho)\} \ar[dr]\ar[dl] & \\\{\pi_{\psi}^+, \pi_{\psi}^-\}    \ar[rr] & & \{\phi_{\psi}(\triv), \phi_{\psi}(\rho)\}}\]

\textbf{Case 2}.   Let $\chi  = 1$.
%There is a unique (up to $G$-conjugacy) homomorphism  $\Psi : \SL_2(\C) \to  G$, called \emph{principal}, such that $\Psi(B_0)$ is contained in exactly one Borel %subgroup of $G$.
The \emph{principal parameter} is the composite map
\[
\phi_0 : \W_K \times  \SL_2(\C)  \to \PSL(2,\C).
\]
defined by extending the \emph{principal} homomorphism $ \SL_2(\C) \to \PSL_2(\C)$ trivially on $\W_K$, is a canonical discrete parameter
for which $\cS_{\phi_0} = 1$.
%, whose centralizer  $S_{\phi_0}$, the centre of  $G$, is as small as possible.
In the local Langlands correspondence for $\cG$, the enhanced parameter $\phi_0(\triv)$ corresponds to the Steinberg representation $\St_{\cG}$, see \cite[6.1.8]{Reed}.

Let $\phi_1$ be the unique parameter for which $\phi_1(\W_K \times \SL_2(\C)) = 1$.   We have
\[
\im\,\phi_1 = 1, \quad C_G(\im\,\phi_1) = G, \quad \cS_{\phi_1} = 1.
\]
 There is a unique enhanced parameter, namely $\phi_1(\triv)$.
We have, at the level of elements, the commutative triangle

\[
\xymatrix{
& \{(1, \triv), (1, \rho)\} \ar[dr]\ar[dl] & \\
\{ \St_{\cG}, 1_{\cG}\}    \ar[rr] & & \{\phi_0(\triv), \phi_1(\triv)\}
}
\]

\textbf{Case 3}. $\chi^2 \neq 1$.   There are no points of reducibility, and we have a commutative triangle of  sets, each with one element:

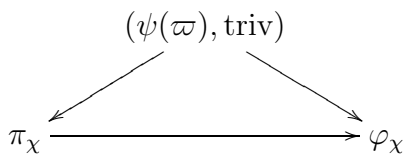
\[\xymatrix{& (\psi(\varpi ), \triv ) \ar[dr]\ar[dl] & \\\pi_{\chi}    \ar[rr] & & \phi_{\chi}}\]

\end{proof}

\begin{cor}\label{parametrization of L-parameters} Let $L/K$ be a quadratic extension of $K$.
The $L$-parameters $\phi_L$ serve as parameters for the $L$-packets in the principal series of $\SL_2(K)$.
\end{cor}

It follows from \S3 that there are countably many $L$-packets in the principal series of $\SL_2(K)$.

\section{The tempered dual}   The following picture

\bigskip

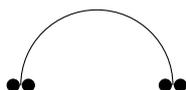
\begin{tikzpicture}\label{pic}
\hskip 4.0cm
\draw (3,0) arc (0:180:1cm);
\draw (0.9,0) circle (0.08cm);
\fill[black] (0.9,0) circle (0.08cm);
\draw (1.1,0) circle (0.08cm);
\fill[black] (1.1,0) circle (0.08cm);
%\draw (1.0,-0.3) node[scale=0.7] {$\pi_{a}$};
\draw (2.9,0) circle (0.08cm);
\fill[black] (2.9,0) circle (0.08cm);
\draw (3.1,0) circle (0.08cm);
\fill[black] (3.1,0) circle (0.08cm);
%\draw (3.0,-0.3) node[scale=0.7] {$\pi_{a_0+a}$};
\end{tikzpicture}

\bigskip
\noindent serves two purposes.   First, it is an accurate portrayal of the extended quotient of the second kind
\[
(\mathbb{T}\q W)_2
\]
where $\T = \{ z \in \C : |z| = 1\}$ and the generator of $W = \Z/2\Z$ acts on  $\T$ sending $z$ to $z^{-1}$.    Secondly, it is (conjecturally) an accurate portrayal of a connected component in the tempered dual of $\cG$.

The topology on $(\T\q W)_2$ comes about as follows.  Let
\[
\Prim ( C(\T) \rtimes W)
\]
  denote the primitive ideal space of the noncommutative
$C^*$-algebra $C(\T) \rtimes W$.   By the classical Mackey theory for semidirect products, we have a canonical bijection
\begin{align}\label{JJJ}
\Prim ( C(\T) \rtimes W) \simeq (\T\q W)_2.
\end{align}
The primitive ideal space on the left-hand side of (\ref{JJJ}) admits the Jacobson topology.   So the right-hand side of (\ref{JJJ}) acquires, by transport of structure, a compact non-Hausdorff topology.  The  picture at the beginning of this section  is intended to portray this topology.   We shall see that the Langlands parameters respect this topology.    The double-points in the picture arise precisely when the corresponding induced representation has length $2$.

\medskip

 The Plancherel Theorem of Harish-Chandra is valid for any local non-archimedean field, see Waldspurger \cite{W}.   This implies that, in the case at hand, the discrete series and the unitary principal series enter into the Plancherel formula.   That is, the tempered dual of $\cG$ comprises the discrete series and the irreducible constituents in the unitary principal series.

We now focus on the case of induced elements.

\medskip

Suppose $\chi^2\neq 1$, with  $\mathfrak{s}=[\cT,\chi]_\cG$.    Let $\psi$ be an unramified unitary character of
$\cT$.   Then we have a natural bijection
\begin{align}\label{T1}
\Irr^{\temp}(G)^{\fs} \simeq \T, \quad \quad \Ind \,\pi_{\psi\chi} \mapsto \psi(\varpi).
\end{align}

\medskip

Suppose $\chi^2 = 1, \chi \neq 1$, with  $\fs=[T,\chi]_G$. Let $W = \Z/2\Z$.   Then we have a bijective map
\begin{align}\label{T2}
\Irr^{\temp}(G)^{\fs} \simeq (\T\q W)_2.
\end{align}
This map is defined as follows.   Let $\rho$ denote the nontrivial character
of $W$.
\begin{itemize}
\item If $\psi^2 \neq 1$, send $\Ind \,\pi_{\psi\chi}$ to $\psi(\varpi)$.
\item   If $\psi = 1$, send the pair of irreducible constituents
$\pi_{\chi}^+, \pi_{\chi}^-$ to the pair of points  $(1, \triv), \; (1,\rho) \in (T\q W)_2$.
\item If $\psi = \epsilon$  the unique unramifed \emph{quadratic} character of $\cT$, send
 the pair of irreducible constituents
$\pi_{\epsilon\chi}^+, \pi_{\epsilon\chi}^-$ to the pair of points  $(- 1, \triv), \; ( - 1,\rho) \in (T\q W)_2$.
\end{itemize}

\medskip

Suppose $\chi=1$  and let $\fs_0=[T,1]_G$. Then we have a continuous bijection which is \emph{not} a homeomorphism:
\begin{align}\label{T3}
\Irr^{\temp}(G)^{\fs} \to  (\T\q W)_2.
\end{align}

\begin{itemize}
\item If $\psi^2 \neq 1$, send $\Ind \,\pi_{\psi\chi}$ to $\psi(\varpi)$.
\item   If $\psi = 1$, send the irreducible representations $\triv_{\cG}, \St_{\cG}$
 to the pair of points  $(1, \triv), \; (1,\rho) \in (T\q W)_2$.
\item If $\psi = \epsilon$   send
 the pair of irreducible constituents
$\pi_{\epsilon}^+, \pi_{\epsilon}^-$ to the pair of points  $(- 1, \triv), \; ( - 1,\rho) \in (T\q W)_2$.
\end{itemize}

\bigskip

By  proposition \ref{compositum quadratic ext} and the above argument, we may represent that part of the tempered dual $\Irr^{\temp}(\SL_2(\mathbb{F}_q((\varpi))))$ which corresponds to the unitary principal series in a diagram along the lines of \cite[p.418]{Pl2}.

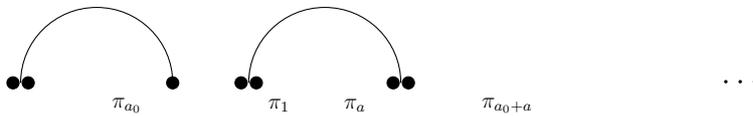
\bigskip

\begin{tikzpicture}
\hskip 2.0cm
\draw (0,0) arc (0:180:1cm);
\draw (-2.1,0) circle (0.08cm);
\fill[black] (-2.1,0) circle (0.08cm);
\draw (-1.9,0) circle (0.08cm);
\fill[black] (-1.9,0) circle (0.08cm);
\draw (-2,-0.3) node[scale=0.7] {$\pi_{a_0}$};
\draw (0,0) circle (0.08cm);
\fill[black] (0,0)circle (0.08cm);
\draw (0.0,-0.3) node[scale=0.7] {$\pi_1$};
%%%%%%%%%%%%%%%%%%%%%%%%%%%%%%%%%%%%%%%%%%%%%%%%%%%%%%%%%%%%%
%%%%%%%% 1st 1/2 semi-circle 2 %%%%%%%%%
\draw (3,0) arc (0:180:1cm);
\draw (0.9,0) circle (0.08cm);
\fill[black] (0.9,0) circle (0.08cm);
\draw (1.1,0) circle (0.08cm);
\fill[black] (1.1,0) circle (0.08cm);
\draw (1.0,-0.3) node[scale=0.7] {$\pi_{a}$};
%%%%%%% 2nd 1/2 semi-circle 2 %%%%%%%%%
\draw (2.9,0) circle (0.08cm);
\fill[black] (2.9,0) circle (0.08cm);
\draw (3.1,0) circle (0.08cm);
\fill[black] (3.1,0) circle (0.08cm);
\draw (3.0,-0.3) node[scale=0.7] {$\pi_{a_0+a}$};
\draw (5.5,0) node {$\cdots$};
\end{tikzpicture}

\bigskip

The first double point represent the $L$-packet $\{\pi_{a_0}^+ , \pi_{a_0}^-\}$. The second and third double-points represent, respectively, the $L$-packets
$\{\pi_a^+ , \pi_a^-\}$ and $\{\pi_{a_0+a}^+ , \pi_{a_0+a}^-\}$. The second half-circle is repeated countably many times, and is parametrized by $L$-parameters $\{\phi_a\}_{a+\wp(K)}$, see theorem \ref{parametrization of L-parameters}.

\textsc{Topology on the tempered dual}.    Let $\cG = \SL_2(\F_q(\varpi)))$.   The tempered dual of $\cG$ is the disjoint union $X = X_{\cG}$ of the discrete series and the irreducible constituents in the principal series.
We equip $X$ with the following topology $\fT$ : The topology $\fT$ must induce the standard topologies on each point, each copy of $\T$, and each copy (except one) of $(\T\q W)_2$,  \emph{all of which (except one)} must become $\fT$-open sets.   On the exceptional  copy of $(\T\q W)_2$  the Steinberg point $\St_{\cG}$ must be
$\fT$-isolated.
  Then $\fT$ is a locally compact topology on $X$.   It is not Hausdorff.

In the space $X$, each $L$-packet in the unitary principal series will feature as a $\fT$-double-point.

There will be countably many double-points, one for each quadratic extension $K(\alpha)$; \emph{cf.} the diagram in \cite{Pl2} for the tempered dual of
$\SL_2(\Q_p)$ with $ p>2$.   In that diagram, there are just three double-points.    For $\SL_2(\Q_2)$ there would be seven double-points.

Each supercuspidal $L$-packet will feature as four $\fT$-isolated points in $X$.

We conjecture that $\fT$ coincides with the Jacobson topology on the primitive ideal space of the reduced $C^*$-algebra of $\cG$.

\section{Biquadratic extensions of $\F_q((\varpi))$}

Quadratic extensions $L/K$ are obtained by adjoining an $\mathbb{F}_2$-line $D\subset K/\wp(K)$. Therefore, $L=K(\wp^{-1}(D))=K(\alpha)$ where $D=span\{a+\wp(K)\}$, with $\alpha^2-\alpha=a$. In particular, if $a_0$ is integer, the $\mathbb{F}_2$-line $V_0=span\{a_0+\wp(K)\}$ contains all the cosets $a_i+\wp(K)$ where $a_i$ is an integer and so $K(\wp^{-1}(\mathfrak{o}))=K(\wp^{-1}(V_0))=K(\alpha_0)$ where $\alpha_0^2-\alpha_0=a_0$ gives the unramified quadratic extension.

Biquadratic extensions are computed the same way, by considering planes $W=span\{a+\wp(K), b+\wp(K)\}\subset K/\wp(K)$. Therefore, if $a+\wp(K)$ and $b+\wp(K)$ are $\mathbb{F}_2$-linearly independent then $K(\wp^{-1}(W)):=K(\alpha, \beta)$ is biquadratic, where $\alpha^2-\alpha=a$ and $\beta^2-\beta=b$, $\alpha, \beta\in K^s$. Therefore, $K(\alpha, \beta)/K$ is biquadratic if $b-a\not\in\wp(K)$.

A biquadratic extension containing the line $V_0$ is of the form $K(\alpha_0,\beta)/K$. There are countably many quadratic extensions $L_0/K$ containing the unramified quadratic extension. They have ramification index $e(L_0/K)=2$. And there are countably many biquadratic extensions $L/K$ which do not contain the unramified quadratic extension. They have ramification index $e(L/K)=4$.

So, there is a plentiful supply of biquadratic extensions $K(\alpha, \beta)/K$.

\subsection{Ramification}

The space $K/\wp(K)$ comes with a filtration

\medskip

\begin{equation}\label{Filtration K/P(K)}
0\subset_1 V_0\subset_f V_1=V_2\subset_f V_3=V_4\subset_f ...\subset K/\wp(K)
\end{equation}
where $V_0$ is the image of $\mathfrak{o}_K$  and $V_i$ ($i>0$) is the image of $\mathfrak{p}^{-i}$ under the canonical surjection $K\rightarrow K/\wp(K)$. For $K=\mathbb{F}_q((\varpi))$ and $i>0$, each inclusion $V_{2i}\subset_f V_{2i+1}$ is a sub-$\mathbb{F}_2$-space of codimension $f$. The $\F_2$-dimension of $V_n$ is
\begin{equation}\label{F_2 dim.}
dim_{\F_2}V_n=1+\lceil n/2 \rceil,
\end{equation}
where $\lceil x \rceil$ is the smallest integer not less than $x$.

\bigskip

Let $L/K$ denote a Galois extension with Galois group $G$. For each $i\geq -1$ we define the $i^{th}$-ramification subgroup of $G$ (in the lower numbering) to be:
$$G_i=\{\sigma\in G: \sigma(x)-x \in\mathfrak{p}_L^{i+1}, \forall x\in\mathfrak{o}_L\}.$$
An integer $t$ is a \emph{break} for the filtration $\{G_i\}_{i\geq -1}$ if $G_t\neq G_{t+1}$. The study of ramification groups $\{G_i\}_{i\geq -1}$ equivalent to the study of breaks of the filtration.

There is another decreasing filtration with upper numbering $\{G^i\}_{i\geq -1}$ and defined by the Hasse-Herbrand function $\psi=\psi_{L/K}$:
$$G^u=G_{\psi(u)}.$$
In particular, $G^{-1}=G_{-1}=G$ and $G^0=G_0$, since $\psi(0)=0$.

\bigskip

Let $G_2=Gal(K_2/K)$ be the Galois group of the maximal abelian extension of exponent $2$, $K_2=K(\wp^{-1}(K))$. Since $G_2\cong K^{\times}/K^{\times 2}$ (Proposition \ref{K*/K^*p is profinite}), the pairing $K^{\times}/K^{\times 2}\times K/\wp(K)\rightarrow \Z/2\Z$ from (\ref{A.-S- Pairing}) coincides with the pairing $G_2\times K/\wp(K)\rightarrow \Z/2\Z$.

The profinite group $G_2$ comes equipped with a ramification filtration $(G_2^u)_{u\geq -1}$ in the upper numbering, see \cite[p.409]{Da}. For $u\geq 0$, we have an orthogonal relation \cite[Proposition 17]{Da}
\begin{equation}\label{orthogonal}
(G_2^u)^{\bot}=\overline{\mathfrak{p}^{-\lceil u \rceil+1}}=V_{\lceil u \rceil-1}
\end{equation}
under the pairing $G_2\times K/\wp(K)\rightarrow \Z/2\Z$.

\bigskip
Since the upper filtration is more suitable for quotients, we will first compute the upper breaks and then use the Hasse-Herbrand function to compute the lower breaks in order to obtain the lower ramification filtration.

According to \cite[Proposition $17$]{Da}, the positive breaks in the filtration $(G^v)_v$ occur precisely at integers prime to $p$. So, for $ch(K)=2$, the positive breaks will occur at odd integers. The lower numbering breaks are also integers. If $G$ is cyclic of prime order, then there is a unique break for any decreasing filtration $(G^v)_v$ (see \cite{Da}, Proposition $14$). In general, the number of breaks depends on the possible filtration of the Galois group.

Given a plane $W\subset K/\wp(K)$, the filtration (\ref{Filtration K/P(K)}) $(V_i)_i$ on $K/\wp(K)$ induces a filtration $(W_i)_i$ on $W$, where $W_i=W\cap V_i$. There are three possibilities for the filtration breaks on a plane and we will consider each case individually.

\bigskip

\textbf{Case 1 :} $W$ contains the line $V_0$, i.e. $L_0=K(\wp^{-1}(W))$ contains the unramified quadratic extension $K(\wp^{-1}(V_0))=K(\alpha_0)$ of $K$. The extension has residue degree $f(L_0/K)=2$ and ramification index $e(L_0/K)=2$. In this case, there is an integer $t>0$, necessarily odd, such that the filtration $(W_i)_i$ looks like
$$0\subset_1 W_0=W_{t-1}\subset_1 W_{t}=W.$$

By the orthogonality relation (\ref{orthogonal}), the upper ramification filtration on $G=Gal(L_0/K)$ looks like
$$\{1\}=...=G^{t+1}\subset_1G^{t}=...=G^0\subset_1G^{-1}=G$$
Therefore, the upper ramification breaks occur at $-1$ and $t$. The lower ramification breaks can be computed using the Hasse-Herbrand function.
%$$G=G^{-1}=\mathbb{Z}/2\mathbb{Z}\times\mathbb{Z}/2\mathbb{Z} \textrm{ ; } G^0=G^1=...=G^{t}=\mathbb{Z}/2\mathbb{Z} \textrm{ ; } G^{t+1}=\{1\}$$
The table for the index of $G^u$ in $G^0$ is as follows:

\medskip

\[
\begin{tabular}{ c c c }
  %\hline
  $u\in$ & $[0, t]$ & $]t, +\infty[$ \\

  \hline \\

  $G^u=$ & $G^0$ & $\{1\}$

  \\

  $(G^0:G^u)=$ & $1$ & $2$ \\
\end{tabular}
\]

\medskip

We have, $\psi(t)=\int_0^t(G^0:G^u)du=t$, and the lower ramifications breaks occur at $-1$ and $t$. It follows that the \textbf{lower filtration} is

\begin{align}\label{ram1}
G_{-1}=G=\mathbb{Z}/2\mathbb{Z}\times\mathbb{Z}/2\mathbb{Z} \textrm{ ; } G_0=...=G_{t}=\mathbb{Z}/2\mathbb{Z} \textrm{ ; } G_{t+1}=\{1\}
\end{align}

\medskip

The number of such $W$ is equal to the number of planes in $V_t$ containing the line $V_0$ but but not contained in the subspace $V_{t-1}$. Note that this number can be computed and equals the number of biquadratic extensions of $K$ containing the unramified quadratic extensions and with a pair of upper ramification breaks $(-1,t)$, $t>0$ and odd.

\begin{ex}
The number of biquadratic extensions containing the unramified quadratic extension and with a pair of upper ramification breaks $(-1,1)$ is equal to the number of planes in an $1+f$-dimensional $\F_2$-space, containing the line $V_0$. There are precisely
$$1+2+2^2+...+2^{f-1}=\frac{1-2^f}{1-2}=q-1$$
of such biquadratic extensions.
\end{ex}

\bigskip

\textbf{Case 2.1 :} $W$ does not contains the line $V_0$ and the induced filtration on the plane $W$ looks like
$$0=W_{t-1}\subset_2 W_{t}=W$$
for some integer $t$, necessarily odd.

The number of such $W$ is equal to the  number of planes in $V_t$ whose intersection with $V_{t-1}$ is $\{0\}$. Note that, there are no such planes when $f=1$. So, for $K=\F_2((\varpi))$, \textbf{case 2.1} does not occur.

Suppose $f>1$. By the orthogonality relation, the upper ramification ramification filtration on $G=Gal(L/K)$ looks like
$$\{1\}=...=G^{t+1}\subset_2G^{t}=...=G^{-1}=G$$
Therefore, there is a single upper ramification break occur at $t>0$ and necessarily odd. The lower ramification breaks occurs at the same $t$, since we have:

\medskip

\[
\begin{tabular}{ c c c }
  %\hline
  $u\in$ & $[0, t]$ & $]t, +\infty[$ \\

  \hline \\

  $G^u=$ & $G^0$ & $\{1\}$

  \\

  $(G^0:G^u)=$ & $1$ & $2^2$ \\
\end{tabular}
\]

\medskip

and so, $\psi(t)=\int_0^t(G^0:G^u)du=t$, and the lower ramifications breaks occur at $-1$ and $t$. It follows that the \textbf{lower filtration} is

\begin{align}\label{ram2.1}
G_{-1}=G=...=G_{t}=\mathbb{Z}/2\mathbb{Z}\times\mathbb{Z}/2\mathbb{Z} \textrm{ ; } G_{t+1}=\{1\}
\end{align}

\medskip

For $f=1$ there is no such biquadratic extension. For $f>1$, the number of these biquadratic extensions  equals the number of planes $W$ contained in an $\F_2$-space of dimension $1+fi$, $t=2i-1$, which are transverse to a given codimension-$f$ $\F_2$-space.

\bigskip

\textbf{Case 2.2 :} $W$ does not contains the line $V_0$ and the induced filtration on the plane $W$ looks like
$$0=W_{t_1-1}\subset_1 W_{t_1}=W_{t_2-1}\subset_1 W_{t_2}=W$$
for some integers $t_1$ and $t_2$, necessarily odd, with $0<t_1<t_2$.

The orthogonality relation for this case implies that the upper ramification filtration on $G=Gal(L/K)$ looks like
$$\{1\}=...=G^{t_2+1}\subset_1G^{t_2}=...=G^{t_1+1}\subset_1G^{t_1}=...=G$$
The upper ramification breaks occur at odd integers $t_1$ and $t_2$.

Now, index of $G^u$ in $G^0$ is:

\medskip

\[
\begin{tabular}{ c c c c}
  %\hline
  $u\in$ & $[0, t_1]$  & $]t_1, t_2]$ & $]t_2, +\infty[$ \\

  \\

  \hline \\

  $(G^0:G^u)=$ & $1$ & $2$ & $2^2$\\
\end{tabular}
\]

\medskip

The lower breaks occur at
$$\psi(t_1)=\int_0^{t_1}(G^0:G^u)du=t_1 .$$
and at
$$\psi(t_2)=\int_0^{t_2}(G^0:G^u)du=\int_0^{t_1}(G^0:G^u)du+\int_{t_1}^{t_2}(G^0:G^u)du$$
$$\hskip -3.9cm =t_1+2(t_2-t_1)=2t_2-t_1 .$$
In this case, the lower breaks occur at $t_1$ and $2t_2-t_1$, with $0<t_1<t_2$ the odd integers where the upper ramification breaks occur.

\medskip

We conclude that the \textbf{lower filtration} is given by
\begin{align}\label{ram2}
G & =G_0=...=G_{t_1}=\mathbb{Z}/2\mathbb{Z}\times\mathbb{Z}/2\mathbb{Z}\\
G_{t_1+1} & =...=G_{2t_2-t_1}=\mathbb{Z}/2\mathbb{Z} \textrm{ ; } G_{2t_2-t_1+1}=\{1\}
\end{align}

There is only a finite number of such biquadratic extensions, for a given pair of upper breaks (or lower breaks) $(t_1,t_2)$.

\section{Formal degrees}  In this section, we are influenced by the lecture notes of Reeder \cite{Reed}, and the preceding three talks in Washington, DC.
   For $\cG = \SL_2(K)$, the dual group $G =  \SO_3(\C)$  contains a unique (up to conjugacy) subgroup  $J \simeq  \Z/2 \times \Z/2$, whose nontrivial elements are $180$-degree rotations about three orthogonal axes. One can check that the centralizer and normalizer of $J$ are given by
\[
C_G(J) = J, \quad 	N_G(J) = O
\]
 where $O \simeq S_4$ is the rotatation group of the octahedron whose vertices are the unit vectors on the
given orthogonal axes. The quotient  $O/J \simeq  \GL_2(\Z/2)$ is the full automorphism group of $J$.

Each bi-quadratic extension $L/K$ gives a surjective homomorphism
\[
\phi_L : \W_F \to  J
\] which is a discrete parameter with
$S_{\phi_L}	= J$, since $C_G(J) = J$,  and whose conjugacy class depends only on $L$, since $O/J = \Aut(J)$.

Since
\[
|S_{\phi_L}| = 4
\]
the $L$-packet $\Pi_{\phi_L}$ has $4$ constituents.  There are countably many biquadratic extensions, therefore there are countably many  $L$-packets
with $4$ constituents.

None of these packets contains the Steinberg representation $\St_{\cG}$ and so, according to Conjecture 6.1.4 in \cite{Reed}, these are all supercuspidal $L$-packets, each with $4$ elements.

Consider the principal parameter:
\[
\Ad \,\phi_0 : \W_K \times \SL_2(\C) \to \SL_2(\C) \to \PSL_2(\C) \to \Aut(\fsl_2(\C))
\]
  The \emph{adjoint gamma value} is given by
\[
\gamma(\varphi_0) = \frac{q}{1 + q^{-1}}
\]
where $q = 2^f$.

Concerning the adjoint gamma value $\gamma(\varphi)$ we have
\[
\Ad \, \varphi : \W_K \to J \to \SO_3(\C) \to^{\Ad} \Aut(\fso_3(\C))
\]
The adjoint representation of $\SO_3(\C)$ is equivalent to the standard representation of $\SO_3(\C)$ on $\C^3$ and so we replace the above sequence of maps by
\[
\Ad \, \varphi : \W_K \to J \to \SO_3(\C).
\]
For the $L$-function, we have
\[
L(\Ad\, \phi,s) = \frac{1}{1 + q^{-s}}
\]
and so we have
\[
\gamma(\phi) = \frac{2}{1 + q^{-1}}\cdot \varepsilon(\phi)
\]
where
\[
\varepsilon(\phi) = \pm q^{\alpha(\phi)/2}.
\]

Note that we have
\begin{align}
 \left |\frac{\gamma(\phi)}{\gamma(\phi_0)}\right | = \frac{2}{q} \cdot \varepsilon(\phi).
 \end{align}

Now $\alpha(\phi)$ is the Weil-Deligne version of the Artin conductor which is give here by
\[
\alpha(\phi)   =  \sum_{i \geq 0} \frac{\dim(\fg/\fg^{D_i})}{[D_0:D_i]}
\]
 see \cite{Reed}, Reeder's notation.
%\[b(\phi) = \] The Swan conductor follows the lower ramification filtration.

We have to take the cases separately, beginning with (\ref{ram1}).

\textbf{Case 1}:  We have
\[
G_{-1}=G=\mathbb{Z}/2\mathbb{Z}\times\mathbb{Z}/2\mathbb{Z} \textrm{ ; } G_0=...=G_{t}=\mathbb{Z}/2\mathbb{Z} \textrm{ ; } G_{t+1}=\{1\}
\]

We have
\[
\alpha(\phi) =  (1 + t)2
\]
According to Conjecture 6.1(1) in \cite{R}, we have
\[
\Deg(\pi_{\phi_L}(\rho)) = \frac{1}{4}\cdot \left |\frac{\gamma(\phi)}{\gamma(\phi_0)}\right | = \frac{1}{4}\cdot \frac{2}{q} \cdot |\varepsilon(\phi)|   = 2^{t-f}
\]
the \emph{canonical formal degree} of each supercuspidal constituent in the packet $\Pi_{\phi_L}$, i.e. the formal degree  w.r.t. the Euler-Poincar\'e measure on $\cG$.   If we fix the field K, then the formal degree tends to $\infty$ as the  break number $t$ tends to $\infty$.

The least allowed value of $t$ is $t = 1$.   When $t = f = 1$, the canonical formal degree of each element in the packet $\Pi_{\phi_L}$ is equal to $1$.   The lower ramification filtration is

\[
G_{-1}=G=\mathbb{Z}/2\mathbb{Z}\times\mathbb{Z}/2\mathbb{Z} \textrm{ ; } G_0 = G_1=\mathbb{Z}/2\mathbb{Z} \textrm{ ; } G_2=\{1\}
\]
and so, according to 6.1(5 ) in \cite{Reed}, the elements in this packet are not of depth zero.

\medskip

\textbf{Case 2.1}: The lower ramification filtration is
\[
G_{-1} = G = \ldots = G_t = \Z/2\Z \times \Z/2\Z \; ;  \; G_{t+1} = \{1\}
\]
We have
\begin{align*}
\alpha(\phi) & = \sum_{i \geq 0} \frac{\dim(\fg/\fg^{D_i})}{[D_0:D_i]} =  (t + 1)3\\
\end{align*}
According to 6.1(1) in \cite{Reed}, we have
\begin{align*}
\Deg(\pi_{\phi_L}(\rho)) & = \frac{1}{4}\cdot \left |\frac{\gamma(\phi)}{\gamma(\phi_0)}\right | \\
& =  \frac{1}{4}\cdot \frac{2}{q} \cdot | \varepsilon(\phi)| \\
& = \frac{1}{2q} \cdot 2^{\alpha(\phi)/2}\\
& = \frac{1}{2q} \cdot 2^{3(1 + t)/2}\\
&    =  2^{3(1 + t)/2  - f - 1 }
\end{align*}
Note that $t$ is odd, therefore the formal degree is a \emph{rational} number.

\textbf{Case 2.2}: This case admits the following lower ramification filtration:

\begin{align*}
G & =G_0=...=G_{t_1}=\mathbb{Z}/2\mathbb{Z}\times\mathbb{Z}/2\mathbb{Z}\\
G_{t_1+1} & =...=G_{2t_2-t_1}=\mathbb{Z}/2\mathbb{Z} \textrm{ ; } G_{2t_2-t_1+1}=\{1\}
\end{align*}

We have
\begin{align*}
\alpha(\phi) & = \sum_{i \geq 0} \frac{\dim(\fg/\fg^{D_i})}{[D_0:D_i]} =  (t_1 + 1)3 +  \frac{(2t_2)2}{2} = 3 + 3t_1 + 2t_2\\
\end{align*}
and, according to 6.1(1) in \cite{Reed}, we have
\begin{align*}
\Deg(\pi_{\phi_L}(\rho)) & = \frac{1}{4}\cdot \left |\frac{\gamma(\phi)}{\gamma(\phi_0)}\right | \\
& =  \frac{1}{4}\cdot \frac{2}{q} \cdot | \varepsilon(\phi)| \\
& = \frac{1}{2q} \cdot 2^{\alpha(\phi)/2}\\
& = \frac{1}{2q} \cdot 2^{3(1 + t_1)/2 + t_2}\\
&    =  2^{3(1 + t_1)/2 + t_2  - f - 1 }
\end{align*}
the canonical formal degree of each supercuspidal  in the packet $\Pi_{\phi_L}$.   If we fix $f$, then the formal degree  tends to $\infty$ as the  break numbers tend to $\infty$.

Note that $t_1$ is odd, therefore all the formal degrees are \emph{rational} numbers, in conformity with the rationality of the gamma ratio
\cite[Prop. 4.1]{GR}.

\end{document}